\documentclass{amsart}
\usepackage{amssymb}
\usepackage{ifthen}
\usepackage{graphicx}

\newtheorem{thm}{Theorem}

\newtheorem{lem}{Lemma}

\newtheorem{conj}{Conjecture}
\theoremstyle{definition}

\newcommand{\G}{{\mathcal G}}
\newcommand{\F}{{\mathcal F}}
\newcommand{\A}{{\mathcal A}}
\newcommand{\R}{{\mathcal R}}

\newcommand{\es}{{\mathcal S}}

\newcommand{\D}{{\mathbb D}}
\newcommand{\real}{{\operatorname{Re}\,}}

\def\be{\begin{equation}}
\def\ee{\end{equation}}

\begin{document}
\title[Third Hankel determinant for Ozaki close-to-convex]{Improved upper bound of third order Hankel determinant for Ozaki close-to-convex functions}

\author[M. Obradovi\'{c}]{Milutin Obradovi\'{c}}
\address{Department of Mathematics,
Faculty of Civil Engineering, University of Belgrade,
Bulevar Kralja Aleksandra 73, 11000, Belgrade, Serbia}
\email{obrad@grf.bg.ac.rs}

\author[N. Tuneski]{Nikola Tuneski}
\address{Department of Mathematics and Informatics, Faculty of Mechanical Engineering, Ss. Cyril and Methodius
University in Skopje, Karpo\v{s} II b.b., 1000 Skopje, Republic of North Macedonia.}
\email{nikola.tuneski@mf.edu.mk}

\subjclass[2000]{30C45, 30C50}
\keywords{analytic, univalent, Hankel determinant, upper bound, Ozaki close-to-convex.}

\begin{abstract}
In this paper we  improve the upper bound of the third order Hankel determinant for the class of Ozaki close-to-convex functions. The sharp bound is conjectured.
\end{abstract}


\maketitle

\section{Introduction and preliminaries}

\medskip

Univalent functions are functions which are analytic, one-on-one and onto on a certain domain. Their study for more than a century shows that problems are significantly more difficult to be solved over the general class  instead of its subclasses.  This is also the case for the upper bound of the Hankel determinant, a problem rediscovered and extensively studied in recent years. Over the class $\A$ of functions $f(z)=z+a_2z^2+a_3z^3+\cdots$ analytic on the unit disk, this determinant is defined by
\[
        H_{q}(n) = \left |
        \begin{array}{cccc}
        a_{n} & a_{n+1}& \ldots& a_{n+q-1}\\
        a_{n+1}&a_{n+2}& \ldots& a_{n+q}\\
        \vdots&\vdots&~&\vdots \\
        a_{n+q-1}& a_{n+q}&\ldots&a_{n+2q-2}\\
        \end{array}
        \right |,
\]
where $q\geq 1$ and $n\geq 1$. The second order Hankel determinants is
\[ H_2(2) =  \left |
        \begin{array}{cc}
        a_2 & a_3\\
        a_3 & a_4\\
        \end{array}
        \right | = a_2a_4-a_{3}^2,
\]
and the third order one is
\[ H_3(1) =  \left |
        \begin{array}{ccc}
        1 & a_2& a_3\\
        a_2 & a_3& a_4\\
        a_3 & a_4& a_5\\
        \end{array}
        \right | = a_3(a_2a_4-a_{3}^2)-a_4(a_4-a_2a_3)+a_5(a_3-a_2^2).
\]

\medskip

For the general class $\es$ of univalent functions in the class $\A$ tehre are very few results concerning the Hankel determinant.  The best known for the second order case  is due to Hayman (\cite{hayman-68}), saying  that $|H_2(n)|\le An^{1/2}$, where $A$ is an absolute constant, and that this rate of growth is the best possible. Another one is \cite {OT-S}, where it was proven that $|H_{2}(2)|\leq  A$, where $1\leq A\leq \frac{11}{3}=3,66\ldots$ and $|H_{3}(1)|\leq  B$, where $\frac49\leq B\leq \frac{32+\sqrt{285}}{15} = 3.258796\cdots$.

\medskip

There are much more results for the subclasses of $\es$. Namely, for starlike functions the upper bounds for the second and the third order Hankel determinant are 1 (\cite{janteng-07}) and $= 0.777987\ldots$ (\cite{MONT-2019-3}), respectively, while for the same bounds for the convex functions they are $1/8$ (\cite{janteng-07}) and $\frac{4}{135}=0.0296\ldots$ (\cite{Kowalczyk-18}). The estimates for the second order case  are sharp, while of the third order are not, but are best known.
For the class $\R\subset\A$ of functions with bounded turning satisfying $\real f'(z)>0$, $z\in\D$, we have
sharp estimate $|H_2(1)|\le \frac{4}{9}  = 0.444\ldots,$ (\cite{janteng-06}) and probably non-sharp  $ |H_3(1)| \le \frac{207}{540} = 0.38333\ldots$ (\cite{OT-R}).

\medskip

In this paper we study two classes introduced by Ozaki.

\medskip

The first one is the class of Ozaki close-to-convex functions
\[ \F = \{f\in\A: \real \left[1+\frac{zf''(z)}{f'(z)}\right]>-\frac12,\, z\in\D\}\]
introduced by Ozaki in 1941 (\cite{ozaki-1941}) and it is a subclass of the class of close-to-convex functions. For this class the non-sharp estiamtes are known $ |H_{2}(2)|\leq \frac{21}{64}$ (\cite{MONT-2018-1}) and $ |H_{3}(1)|\leq \frac{180+69\sqrt{15}}{32\sqrt{15}}=3.6086187\ldots$ (\cite{ind-1}). We will significantly improve the second estimate to the value $0.1375\ldots$. More about this class one can find in \cite[Sect. 9.5]{DTV-book}.

\medskip

The other class that we will be considered is
\[ \G = \{f\in\A: \real \left[1+\frac{zf''(z)}{f'(z)}\right]<\frac32,\, z\in\D\},\]
Ozaki in \cite{ozaki-1941} introduced this class and proved that it is subclass of $\es$.
Later, Sakaguchi in \cite{saka} and R. Singh and S. Singh in \cite{singh} showed, respectively,  that functions in $\mathcal{G}$ are close-to-convex and starlike.
Again in \cite{MONT-2018-1} it was shown that $ |H_{2}(2)|\leq \frac{9}{320}=0.028125\ldots$. Here we will give estimate of the third Hankel determinant.

\medskip

In the studies given in this paper we use approach based on the estimates of the coefficients of  Shwartz function due to Prokhorov and Szynal (Lemma \ref{lem-prok} given below). This approach is essentially different than the  commonly used and  is the main reason for the improvement in the estimate for the class $\F$ mentioned above. Usually the research is done using a result on coefficients of Carath\'{e}odory functions (functions from with positive real part on the unit disk)  that involves Toeplitz determinants (see \cite[Theorem 3.1.4, p.26]{DTV-book} and \cite{granader}).

\medskip

Here is the result of Prokhorov and Szynal that we will need. In more general form it can be found in \cite[Lemma 2]{Prokhorov-1984}.

\begin{lem}\label{lem-prok}
Let $\omega(z)=c_{1}z+c_{2}z^{2}+\cdots $ be a Schwarz function, i.e., be analytic in the unit dick and $|\omega(z)|<1$ when $z\in\D$ and  $\mu$ and $\nu$ be real numbers.
If $\frac12\le|\mu|\le2$ and $\frac{4}{27}(|\mu|+1)^3-(|\mu|+1)\le\nu\le1$, then
$\left|c_{3}+\mu c_{1}c_{2}+\nu c_{1}^{3}\right|\leq 1.$
\end{lem}

We will also need the following,  almost forgotten result of Carleson (\cite{carlson}).

\begin{lem}\label{lem-carl}
Let $\omega(z)=c_{1}z+c_{2}z^{2}+\cdots $ be a Schwarz function. Then
\[|c_2|\le1-|c_1|^2 \quad\mbox{and}\quad |c_4|\le1-|c_1|^2 -|c_2|^2. \]
\end{lem}

\medskip

\section{Main results}

\smallskip

We begin with improvement of the upper bound of the third Hankel determinant for the class $\F$ of Ozaki close-to-convex functions.

\begin{thm}\label{main-thm}
Let $f\in\F$ is of the form $f(z)=z+a_2z^2+a_3z^3+\cdots$. Then
\[ |H_3(1)| \le \frac{1}{8} = 0.125. \]
\end{thm}

\begin{proof}
For a function $f\in\F$ there exists a Schwarz function $\omega(z) = c_1z+c_2z^2+\cdots$ such that
\be\label{e4}
1+\frac{zf''(z)}{f'(z)} = -\frac12+\frac32\cdot \frac{1+\omega(z)}{1-\omega(z)},
\ee
i.e.,
\[ [zf'(z)]' \cdot [1-\omega(z)] = [1+2\omega(z)]\cdot f'(z). \]
By equating the coefficients in the abovr expression we receive
\be\label{e6}
\begin{split}
a_2 &= \frac32c_1,\\
a_3 &= \frac12(4c_1^2+c_2),\\
a_4 &= \frac12(2c_3+13c_1c_2+20c_1^3),\\
a_5 &= \frac{3}{40}(2c_4+12c_1c_3+46c_1^2c_2+40c_1^4+5c_2^2).
\end{split}
\ee
Using \eqref{e6} we have
\[
H_3(1) =
\frac{1}{320}  \left[4 c_1 ^4  c_2 +8  c_1 ^3  c_3 + 4  c_1   c_2   c_3 -23 c_1^2 c_2 ^2 - 12c_1^2 c_4 + 20c_2^3 -20c_3^2 +24  c_2   c_4 \right]
\]
and
\[
\begin{split}
320 H_3(1)
&= -20\left[c_3^2-\frac15c_1c_2c_3+\left(\frac{1}{10}\right)^2c_1^2c_2^2\right] + \frac{1}{5}c_1^2c_2^2 - 23c_1^2c_2^2 \\
&+ 8c_1^3\left(c_3+\frac12c_1c_2\right)+20c_2^3+12c^4(2c_2-c_1^2).
\end{split}
\]
From here
\be\label{e7}
\begin{split}
320|H_3(1)| &\le
20 \left|c_3-\frac{1}{10}c_1c_2\right|^2 + \frac{114}{5}|c_1|^2|c_2|^2  +8|c_1|^3 \left| c_3+\frac12c_1c_2 \right|\\
& +20|c_2|^3+ 12\left(2 |c_2| + |c_1|^2\right)|c_4|.
\end{split}
\ee

\medskip

By applying Lemma \ref{lem-prok} (with $(\mu,\nu)=(1/10,0)$ and $(\mu,\nu)=(1/2,0)$) and Lemma \ref{lem-carl}, we receive
\[
\begin{split}
320|H_3(1)| &\le
20  + \frac{114}{5}|c_1|^2|c_2|^2  +8|c_1|^3  +20|c_2|^3\\
&\quad + 12\left(2 |c_2| + |c_1|^2\right)(1-|c_1|^2-|c_2|^2)\\
&= \frac{54}{5}|c_1|^2 |c_2|^2 +8 |c_1|^3 -4 |c_2|^3 +24|c_2|\\
&\quad -24 |c_1|^2|c_2| +12 |c_1|^2 -12 |c_1|^4\\
&= 20  + h(|c_1|,|c_2|),
\end{split}
\]
where
\[ h(x,y) =  \frac{54}{5}x^2 y^2 +8 x^3 -4 y^3 +24 y -24 x^2 y +12 x^2 -12 x^4,\]
$0\le x\le1$ and $0\le y \le1-x^2$.

\medskip

We continue with finding the maximum of the function $h$ on the region $\Omega=\{(x,y): 0\le x\le1, 0\le y\le1-x^2\}$.

\medskip

The function $h$ has no critical points in the interior of $\Omega$ because $h'_y(x,y) = x^2\left(\frac{108}{5} y-24 \right)-12 y^2+24 = 0$ has only one positive solution for $x$, that is $\sqrt{\frac{5(2-y^2)}{10-9y}}$, an increasing function of $y$ over $(0,\infty)$ with $x(0)=1$.

\medskip

Therefore, we continue studying $h$ on the edges of $\Omega$.

\medskip

For $x=0$, $h(0,y)=24y-4y^3 \le h(0,1)=20$.

For $x=1$, we have $y=0$, and $h(1,0)=8$.

For $y=0$, $h(x,0)=x^2(-12 x^2+8 x+12) $ which can be easily shown to increasing function on the segment $[0,1]$, with maximal value $h(1,0)=8$.

Finally, $g(x):=h(x,1-x^2)=\frac{74}{5} x^6-\frac{108}{5}x^4+8 x^3-\frac{66 }{5}x^2+20$ is a decreasing function on the interval $[0,1]$, since $g'(x)=-\frac{12}{5}x(1-x)(11 + x + 37 x^2 + 37 x^3)$ and $g'(x)=0$ has no solutions on $(0,1)$. Thus, $h(x,1-x)\le g(0)=20$.

\medskip

The above analysis bring the final conclusion that $h$ has maximal value 20 on $\Omega$, i.e.,
\[ |H_3(1)| \le \frac{1}{320}(20+ 20) = \frac{1}{8}.\]
\end{proof}

\medskip

The previous result, although significantly improves the one from \cite{ind-1}, still is not sharp, as the following one dealing with the class $\G$.

\medskip

\begin{thm}\label{th2}
Let $f\in\G$ and is of the form $f(z)=z+a_2z^2+a_3z^3+\cdots$. Then
\[ |H_3(1)| \le \frac{17}{1080} = 0.01574\ldots. \]
\end{thm}

\medskip

\begin{proof}
Similarly as in the proof of the previous theorem, for each function $f$ from $\G$, there exists a function  $\omega(z)= c_1z+c_2z^2+\cdots$, analytic in $\D$, such that $|\omega(z)|<1$ for all $z$ in $\D$, and
\be\label{eeq}
1+\frac{zf''(z)}{f'(z)} = \frac32-\frac12\cdot \frac{1+\omega(z)}{1-\omega(z)},
\ee
i.e.,
\[   [zf'(z)]' \cdot [1-\omega(z)] = [1-2\omega(z)]\cdot f'(z). \]
From here, by equating the coefficients we receive
\[
\begin{split}
a_2 &= -\frac{1}{2}c_1,\\
a_3 &= -\frac16c_2,\\
a_4 &= -\frac{1}{24}(2c_3+c_1c_2),\\
a_5 &= -\frac{1}{120}(6c_4+4c_1c_3+3c_2^2+2c_1^2c_2).
\end{split}
\]
From here, after some calculations we receive
\[
\begin{split}
H_3(1) &=
 \frac{1}{8640}  \left[-60c_3^2 - 132 c_1c_2c_3 + 72c_1^3c_3 + 36c_4(2c_2+3c_1^2) \right.\\
&\quad \left. + 36c_1^4c_2 + 76c_2^3 + 3c_1^2c_2^2 \right],
\end{split}
\]
i.e.,
\[
\begin{split}
8640 H_3(1) &=
-60\left[ c_3^2 +\frac{11}{5} c_1c_2c_3 + \left(\frac{11}{10}\right)^2c_1^2c_2^2\right] + \left[60\left(\frac{11}{10}\right)^2+3\right]c_1^2c_2^2\\
&+ 72c_1^3\left(c_3+\frac12c_1c_2\right)+76c_2^2+36(2c_2+3c_1^2)c_4
\end{split}
\]
and further
\[
\begin{split}
8640 |H_3(1)| &=
60\left| c_3 + \frac{11}{10} c_1c_2\right|^2 + \frac{756}{10}|c_1|^2|c_2|^2+ 72|c_1|^3\left|c_3+\frac12c_1c_2\right|\\
&+76|c_2|^2+36(2|c_2|+3|c_1|^2)|c_4|.
\end{split}
\]
In a similar way as in the proof of the previous theorem, from  Lemma \ref{lem-prok} (with $(\mu,\nu)=(11/10,0)$ and $(\mu,\nu)=(1/2,0)$) and Lemma \ref{lem-carl}, we receive
\[
\begin{split}
8640 |H_3(1)| &=
60 + \frac{756}{10}|c_1|^2|c_2|^2+ 72|c_1|^3+76|c_2|^2\\
&\quad +36(2+|c_1|^2)(1-|c_1|^2-|c_2|^2)\\
&= 60 + h(|c_1|,|c_2|),
\end{split}
\]
where
\[h(x,y) = \frac{756}{10} x^2 y^2 + 72 x^3 + 76 y^3 +
36 \left(2+x^2\right) \left(1-x^2-y^2\right), \]
$(x,y)\in\{(x,y): 0\le x\le 1, 0\le y\le1-x^2\}=:\Omega$.
\medskip

Since $h'_x(x,y)=\frac{36}{5} x \left(10 x (3-2 x)+11 y^2-10\right)$, then $h'_x(x,y)=0$ has positive solution $y_*(x) = \sqrt{\frac{10}{11}(2 x^2-3 x+1)}$ for $x\in(0,1/2)$. Further, $h'_y(x,y)=\frac{12}{5} y \left(33 x^2+95 y-60\right)$ and
\[
\begin{split}
&g(x)=h'_y(x,y_*(x)) \\
&= \frac{24}{\sqrt{55}} \sqrt{(x-1)\left(x-\frac12\right)} \left[33 x^2+95 \sqrt{\frac{20}{11}} \sqrt{(x-1)\left(x-\frac12\right)}-60\right]
\end{split}
\]
on the interval $(0,1)$, has solutions $x_1=0.5$ and $x_2=0.2311\ldots$, with $y_1=g(x_1)=0.75$ and
$y_2=g(x_2)= 0.6130\ldots$, respectively. Both, $(x_1,y_1)$ and $(x_2,y_2)$ are in $\Omega$, so are critical points of $h$ in the interior of $\Omega$, such that $h(x_1,y_1)=69.75$ and  $h(x_2,y_2)=62.10899\ldots$.

\medskip

Further, on the edges of $\Omega$ we have the following.

\medskip

For $x=0$, $h(0,y)=76 y^3-72y^2+72\le h(0,1)=76$.

For $x=1$, $h(0,1)=72$.

For $y=0$, $h(x,0)=-36 x^4+72 x^3-36 x^2+72 \le h(0,0)=h(1,0)=72$.

For $y=1-x^2$, we have $h(x,1-x^2)=-\frac{182 x^6}{5}+\frac{204 x^4}{5}+72 x^3-\frac{402 x^2}{5}+76 \le 76$ obtained for $x=0$.

\medskip

All the analysis from above leads to the conclusion that $h$ has maximal value 76 on $\Omega$ obtained for $x=0$ and $y=1$, i.e.,
\[ |H_3(1)| \le \frac{1}{8640}(60+76) = \frac{17}{1080} = 0.01574\ldots. \]
\end{proof}

\medskip

The estimates of the third Hankel determinant given in Theorem \ref{main-thm} and Theorem \ref{th2} are probably not sharp. Here is a conjecture of the sharp values.

\begin{conj} Let $f\in\A$ and is of the form $f(z)=z+a_2z^2+a_3z^3+\cdots$.
\begin{itemize}
\item[$(i)$] If $f\in\F$, then $|H_{3}(1)|\leq \frac{1}{16} = 0.0625$;
\item[$(ii)$] If $f\in\G$, then $|H_{3}(1)|\leq \frac{19}{2160}=0.00879\ldots$.
\end{itemize}
Both estimates are sharp with extremal functions $\frac{1+2z^2}{1-z^2}$ and $\frac{1}{2} \left(z \sqrt{1-z^2}+\arcsin{z}\right)$, respectively,  obtained for $\omega(z)=z^2$ in \eqref{e4} and \eqref{eeq}.
\end{conj}

\end{document}